\documentclass[12pt]{amsart}

\usepackage{t1enc}
\usepackage[cp1251]{inputenc}
\usepackage[ukrainian,english]{babel}

\usepackage[all]{xy}
\usepackage{bbm}

\numberwithin{equation}{section}
\newtheorem{theorem}{Theorem}[section]
\newtheorem{prop}[theorem]{Proposition}
\newtheorem{corol}[theorem]{Corollary}
\theoremstyle{definition}
\newtheorem{defin}[theorem]{Definition}
\theoremstyle{remark}

\def\tO{\tilde{\mathcal O}}		\def\tX{\tilde X}	
\def\tkC{\tilde{\mathcal C}}		\def\tS{\tilde{\mathcal S}}	\def\tio{\tilde\iota}
\def\hpi{\bar\pi}			\def\tbS{\tilde{\mathbf S}}	\def\tA{\tilde{\mathcal A}}
\def\hT{\overline{\mathcal T}}
\def\lb{\textup{(}}	\def\rb{\textup{)}}
\def\bla{\boldsymbol{\la}}

      \def\Th{\Theta}

       \def\th{\theta}
\def\al{\alpha}       \def\de{\delta}
        \def\eps{\varepsilon}
       
\def\la{\lambda}      
      \def\io{\iota}

\def\mA{\mathbb A} \def\mN{\mathbb N}
 
 \def\mP{\mathbb P}

\def\mI{\mathbb I}

 \def\mZ{\mathbb Z}

\def\aK{\mathbbm k}

\def\dB{\mathfrak B}

\def\dE{\mathfrak E} 
\def\dF{\mathfrak F}

 \def\dX{\mathfrak X}

\def\kA{\mathcal A} \def\kN{\mathcal N}
 \def\kO{\mathcal O}
\def\kC{\mathcal C}

\def\kF{\mathcal F} \def\kS{\mathcal S}
\def\kG{\mathcal G} \def\kT{\mathcal T}

\def\kJ{\mathcal J} 
\def\kK{\mathcal K} 
\def\kL{\mathcal L}

 \def\bN{\mathbf N}

\def\bF{\mathbf F} \def\bS{\mathbf S}
\def\bG{\mathbf G}

 \def\fP{\mathbf p}

\def\fK{\mathbf k} 
 
\def\fM{\mathbf m} 

\def\b1{\mathbf 1}		\def\b1{\mathbf 2}
\def\b1{\mathbf 3}		\def\b1{\mathbf 4}
\def\b1{\mathbf 5}		\def\b1{\mathbf 6}
\def\b1{\mathbf 7}		\def\b1{\mathbf 8}
\def\b1{\mathbf 9}		\def\b1{\mathbf 0}

\def\rA{\mathrm A}

\def\lb{\textup{(}}	\def\rb{\textup{)}}

\def\sb{\subset}		\def\sbe{\subseteq}
\def\xx{\times}		\def\*{\otimes}
\def\+{\oplus}		\def\bop{\bigoplus}

\def\cen{\mathop\mathrm{center}\nolimits}

\def\len{\mathop\mathrm{length}\nolimits}
\def\sng{\mathop\mathrm{sg}\nolimits}
\def\tsg{\mathop\mathrm{\widetilde{sg}}\nolimits}
\def\vb{\mathop\mathrm{VB}\nolimits}

\def\hom{\mathop\mathrm{Hom}\nolimits}
\def\coh{\mathop\mathrm{Coh}\nolimits}
\def\End{\mathop\mathrm{End}\nolimits}
\def\spec{\mathop\mathrm{spec}\nolimits}
\def\Mat{\mathop\mathrm{Mat}\nolimits}
\def\rad{\mathop\mathrm{rad}\nolimits}
\def\END{\mathop{\mathcal E\!\mathit{nd}}\nolimits}
\def\Hom{\mathop{\mathcal H\!\mathit{om}}\nolimits}
\def\rep{\mathop\mathrm{rep}}
\def\rk{\mathop\mathrm{rk}\nolimits}
\def\et{\mathop\mathrm{ex}}
\def\supp{\mathop\mathrm{supp}}

\def\NC{noncommutative curve}
\def\nnc{noncommutative nodal curve}
\def\NV{noncommutative variet}
\def\VB{vector bundle}
\def\oc{one-to-one correspondence }
\def\iff{if and only if }

\def\set#1{\left\{\,#1\,\right\}}
\def\setsuch#1#2{\left\{\,#1\mid #2\,\right\}}
\def\row#1#2{\left( #1_1 , #1_2 , \dots , #1_{#2} \right)}
\def\lst#1#2{#1_1 , #1_2 , \dots , #1_{#2}}
\def\mtr#1{\begin{pmatrix}#1\end{pmatrix}}

\def\ito{\stackrel\sim\to}	\def\mps{\mapsto}
\def\esim{\approx}

\title{Vector bundles over noncommutative\\ nodal curves}
\author{Yuriy A. Drozd}
\author{Denys E. Voloshyn}
\address{Institute of Mathematics, Tereschenkivska 3, 01601 Kiev, Ukraine}
\email{y.a.drozd@gmail.com, drozd@imath.kiev.ua}
\email{denys\_vol@ukr.net}
\urladdr{www.imath.kiev.ua/$\sim$drozd}

\begin{document}

\begin{abstract}
  We describe \VB s over a class of \NC s, namely, over \nnc s of string type and of 
almost string type. We also prove that in other cases the classification of \VB s over a 
\NC\ is a wild problem.
\\
\\
{\selectlanguage{ukrainian}
  Описано векторні розшарування над деяким класом некомутативних кривих, а саме, над 
нодальними некомутативними кривими струнного та майже струнного типу. Встановлено також, 
що в інших випадках класифікація векторних розшарувань над некомутативною кривою є дикою 
задачею.}
\end{abstract}

{\selectlanguage{ukrainian}
 УДК 512.723 }
\\

\maketitle


 Classification of \VB s over algebraic curves is a popular topic in modern 
mathematical literature. It is due to their importance for many branches of mathematics 
and mathematical physics. Vector bundles over the projective line were described by 
Birkhoff 
\cite{bi} and Grothendieck \cite{gr}, \VB s over elliptic curves were classified by 
Atiyah \cite{at}. In the paper \cite{dg} Greuel and the first author described \VB s over 
a class of singular curves (line configurations of types $\rA$ and $\tilde{\rA}$) 
and showed that in all other cases a complete classification of \VB s is a ``wild 
problem'' in the sense of representation theory of algebras. 

 This paper is devoted to analogous questions for \emph{noncommutative curves}. Perhaps, 
the first results in this direction were obtained by Geigle and Lenzing \cite{gl} who 
considered the so called \emph{weighted projective lines}. Though the original definition 
of this paper was in the frames of ``usual'' (commutative) algebraic geometry, these 
curves are actually of noncommutative nature. They can be considered as such 
noncommutative curves that the underlying algebraic curve is a projective line and all 
localizations of the structure sheaf are hereditary. In some sense, it is the simplest 
example of noncommutative curves, though their theory is far from being simple.

 We consider the ``next step,'' namely the case when the localizations of the structure 
sheaf are \emph{nodal} in the sense of \cite{bd}. In particular, this class contains all 
line configurations in the sense of \cite{dg}. We reduce the description of \VB s over 
such curves to the study of a bimodule category in the sense of \cite{ds,dg}. Using this 
reduction, we describe \VB s in two cases: \emph{string type} and \emph{almost string 
type}, see Sections \ref{s3} and \ref{s4}. Note that the string type is an immediate 
generalization of line configurations of types $\rA$ and $\tilde\rA$.
 The main tool in this description is a special 
sort of bimodule problems, namely, the so called \emph{bunches of chains}. Fortunately, 
these problems are well elaborated and a good description of representations is given in 
\cite{bon}. We also show that in all other cases the classification of \VB s is a wild 
problem (Section \ref{s5}). Thus, in some sense, the question about the ``representation 
type'' of the category of \VB s over noncommutative curves is completely solved.

\section{Noncommutative curves, vector bundles\\ and categories of triples}
\label{s1}

 We call a \emph{\NV y} a pair $(X,\kA)$, where $X$ is an algebraic 
variety over an algebraically closed field $\aK$ (reduced, but maybe reducible) and 
$\kA$ is a sheaf of $\kO_X$-algebras which is coherent as a sheaf of 
$\kO_X$-modules. We often speak about a ``\NV y $\kA$'' not mentioning explicitly the 
underlying variety $X$. We denote by $\kK_X$ (or 
$\kK$) the sheaf of total rings of fractions of $\kO_X$ (it is locally constant) and set
$\kK(\kA)=\kA\*_{\kO_X}\kK_X$.  Without loss of generality we may (and usually 
will) suppose that $\kA$ is \emph{central}, i.e. $\kO_{X,x}=\cen(\kA_x)$ for 
each $x\in X$. Otherwise we can replace $X$ by the variety $X'=\spec\kC$, where 
$\kC=\cen(\kA)$. We define a \NC\ as a \NV y $(X,\kA)$ such that $X$ is a curve 
(that is all its components are $1$-dimensional) and $\kA$ is reduced, that is has no 
nilpotent ideals.  A coherent sheaf of $\kA$-modules $\kF$ is said to be a 
\emph{vector bundle} over $(X,\kA)$ if it is \emph{locally projective}, i.e. the 
$\kA_x$-module $\kF_x$ is projective for every $x\in X$. We denote by $\vb(X,\kA)$ or by 
$\vb(\kA)$ the category of vector bundles over $(X,\kA)$.

 We call a \NC\ $(X,\kA)$ \emph{normal} if, for every point $x\in X$, the algebra 
$\kA_x$ is a \emph{maximal $\kO_{X,x}$-order}, that is there is no 
$\kO_{X,x}$-subalgebra $\kA_x\sb \kA'\sb \kK_x$ which is also finitely generated as 
$\kO_{X,x}$-module. Since $\kA$ is reduced, there is a normal curve 
$\tX=(X,\tA)$ such that $\kA\sbe\tA\sb \kK_X$. Moreover, $\kA_x=\tA_x$
for almost all $x\in X$ (it follows from \cite{cr}). We call 
$(X,\tA)$ \emph{a normalization} of $X$ and denote by $\sng\kA$ the set of all points 
$x\in X$ such that $\kA_x\ne\tA_x$. Note that 
such a normalization is, as a rule, not unique, though $\sng\kA$ does not depend on the 
choice of normalization. Let $\tkC=\cen(\tA)$, $\tX=\spec\tkC$. We can (and will) 
consider $\tA$ as a sheaf of central $\kO_{\tX}$-algebras, hence consider the 
normalization as the \NC\ $(\tX,\tA)$. The natural morphism of ringed 
spaces $\pi:(\tX,\tA)\to (X,\kA)$ is defined. We also denote by $\tsg\kA$ the
set-theoretical preimage 
$\pi^{-1}(\sng \kA)$. If $\lst{\tX}s$ are the irreducible components of $\tX$, we set
$\tA_i=\tA|_{\tX_i}$, so consider the \NC s $(\tX_i,\tA_i)$. We also set 
$\tsg_i\kA=\tsg\kA\cap \tX_i$.
 Let $X_i=\pi(\tX_i)$. Certainly, each $X_i$ is an irreducible component of $X$, but 
these components need not be different. We set $\kK_i(\kA)=\kK(\kA)|_{\tX_i}$.
It is a constant sheaf of central $\kK_{X_i}$-gebras.  Since 
$\aK$ is algebraically closed, the Brauer group of the field $\kK_i=\kK_{X_i}$ is trivial 
\cite[Chapter II, \S\,3]{ser}, so $\kK_i(\kA)\simeq\Mat(n_i,\kK_i)$ for some $n_i$. We 
call a \NC\ $(X,\kA)$ \emph{rational} if so is the curve $X$, i.e. all components of 
$\tX$ are isomorphic to the projective line $\mP^1$.

 For calculation of vector bundles over \NC s one can use the 
``sandwich procedure,'' just as it has been done in \cite{dg} in the commutative case. 
Let $\pi:(\tX,\tA)\to (X,\kA)$ be a normalization of a \NC\ $(X,\kA)$. We denote by 
$\kJ$ the \emph{conductor of $\tA$ in $\kA$}, that is the 
maximal sheaf of $\tA$-ideals contained in $\kA$. We consider the \NV ies 
$(\sng\kA,\kS)$ and $(\tsg\kA,\tS)$, where $\kS=\kA/\kJ$ and 
$\tS=\tA/\kJ$. 
 These varieties are $0$-dimensional and usually not reduced. We denote by 
$\hpi:(\tsg\kA,\tS)\to (\sng\kA,\kS)$ the restriction of $\pi$ onto $(\tsg\kA,\tS)$ and 
by $\io$ and $\tio$, respectively, the closed embeddings $(\sng\kA,\kS)\to(X,\kA)$ and 
$(\tsg\kA,\tS)\to(\tX,\tA)$. So we have a commutative diagram of morphisms of \NV ies
\[
  \xymatrix@=1ex{  (\tsg\kA,\tS) \ar[rr]^{\tio} \ar[dd]_{\hpi} && (\tX,\tA) \ar[dd]^{\pi} 
\\  \\			  (\sng\kA,\kS) \ar[rr]^{\io} && (\tX,\tA)	}\]
 Since $(\sng\kA,\kS)$ and $(\tsg\kA,\tS)$ are $0$-dimensional, coherent sheaves on them 
can be identified with finitely generated modules over the algebras, respectively, 
\[
   \bS=\prod_{x\in \sng\kA}\kA_x/\kJ_x\ 	\text{ and }\
 \tbS=\prod_{y\in\tsg\kA}\tA_y/\kJ_y. 
\]

 Following \cite{bd,bd2}, we introduce the \emph{category of triples} $\kT(\kA)$ as 
follows. 
\begin{itemize}
\item 
 The objects of $\kT(\kA)$ are triples $(\kG,P,\th)$, where
\begin{itemize}
\item 
  $\kG$ is a vector bundle over $\tA$,
\item  
 $P$ is a vector bundle over $\kS$, or, the same, a finitely generated projective 
$\bS$-module,
\item  
 $\th$ is an isomorphism $\hpi^*P\to\tio^*\kG$, or, the same, an isomorphism of 
$\tbS$-modules $\tbS\*_\bS P\to \prod_{y\in\tsg\kA}\kG_y/\kJ_y\kG_y$.
\end{itemize}
\item  
 A morphism $(\kG,P,\th)\to(\kG',P',\th')$ is a pair $(\Phi,\phi)$, where 
$\Phi\in\hom_{\tA}(\kG,\kG')$ and $\phi\in\hom_{\kS}(P,P')$ such that the 
induced diagram
\[
  \xymatrix@C=1em@R=1.5em{ \hpi^*P \ar[rr]^{\hpi^*\phi} \ar[d]_\th && \hpi^*P' 
\ar[d]^{\th'} \\	 \tio^*\kG \ar[rr]^{\tio^*\Phi} && \tio^*\kG' 	}  
\]
 is commutative. 
\end{itemize}
 One easily sees that $\kT(\kA)$ is indeed a full subcategory of a \emph{bimodule 
category} in the sense of \cite{ds}, namely, the category defined by the 
$\vb(\kS)\mbox{-}\!\vb(\tA)$-bimodule $\hom_{\tS}(\hpi^*P,\tio^*\kG)$. It can also be 
considered as the push-out of the categories $\vb(\tA)$ and $\vb(\kS)$ over the category 
$\vb(\tS)$ with respect to the functors $\tio^*$ and $\hpi^*$. So it is an analogue of 
Milnor's construction of projective modules from \cite[\S\,2]{mil}.

 We define the functor $\bF:\vb(\kA)\to\kT(\kA)$, which maps a \VB\ $\kF$ to the triple 
$(\pi^*\kF,\io^*\kF,\th_\kF)$, where $\th_\kF$ is the natural isomorphism 
$\hpi^*\io^*\kF\to\tio^*\pi^*\kF$.  The same considerations as in 
\cite{bd2,dg} give the following result.

\begin{theorem}\label{11}
  The functor $\bF$ induces an equivalence of the categories $\vb(\kA)\ito\kT(\kA)$. The 
inverse functor $\bG:\kT(\kA)\to\vb(\kA)$ maps a triple $(\kG,P,\th)$ to the preimage in 
$\kG$ of the $\bS$-submodule $\th(1\*P)\sbe \tio^*\kG$.
\end{theorem}

 \section{Nodal curves}
\label{s2}

 \begin{defin}\label{21}
\begin{enumerate}
\item   
  An algebra $R$ over a local commutative ring $O$ of Krull dimension $1$, which is 
finitely generated and torsion free as $O$-module, is said to be \emph{nodal} 
\cite{bd,vol} if the following conditions hold:
 \begin{enumerate}
\item 
  $\End_R(\rad R)=H$ is hereditary.
\item  
  $\rad H=\rad R$ (under the natural embedding of $R$ into $H$).
\item  
       $\len_R(H\*_RU)\le2$ for every simple $R$-module $U$.
\end{enumerate}
 Note that a nodal algebra never has nilpotent ideals, since it holds for any hereditary 
$O$-algebra.

\item  
   A \NC\ $(X,\kA)$ is said to be \emph{nodal} if every algebra $\kA_x\ (x\in X)$ is a 
nodal $\kO_{X,x}$-algebra.
\end{enumerate}
 If $\kA=\kO_X$, so we deal with a ``usual'' (commutative) curve, it means that  all 
singular points of $X$ are nodes (ordinary double points). 
\end{defin}

 We recall the construction of nodal algebras over the ring $O=\aK[[t]]$ from \cite{vol}. 
Up to Morita equivalence such algebra is given by a tuple $\bN=(s;\lst ns;\sim)$, where 
$s$ and $\lst ns$ are positive integers, while $\sim$ is a symmetric relation on the set 
of pairs $\mI=\setsuch{(k,i)}{1\le k\le s,\,1\le i\le n_k}$ satisfying the following 
conditions:
\begin{enumerate}
\item[(N1)]\ 
 $\#\setsuch{(l,j)\in\mI}{(l,j)\sim(k,i)}\le1$ for each pair $(k,i)\in\mI$.
\item[(N2)]\   
 If $(k,i)\sim(k,i)$, then $i<n_k$ and $(k,i+1)\not\sim(l,j)$ for any $(l,j)\in\mI$.
\end{enumerate}
 Namely, define $R(\bN)$ as the subring of $M(\bN)=\prod_{k=1}^s\Mat(n_k,O)$ consisting 
of such collections of matrices $\row As$, where $A_k=(a_{ij}^k)\in\Mat(n_k,O)$, that   
\begin{align}
&  a^k_{ij}\equiv 0 \hskip-1.5ex\pmod t \,\text{ if }\, i>j \,\text{ or }\, i=j-1 
\,\text{ and }\, (k,i)\sim(k,i), \label{e21}\\
&  a^k_{ii}\equiv a^l_{jj} \hskip-1.5ex\pmod t \, \text{ if }\, (k,i)\sim(l,j). 
\label{e22}
\end{align}

\begin{theorem}[\cite{vol}]\label{22}
  \begin{enumerate}
\item 
  Every ring $R(\bN)$ is a nodal $O$-algebra.
\item  
 Every nodal $O$-algebra is Morita equivalent to one of the rings $R(\bN)$.
\item  
 $\rad R(\bN)$ consists of such collections $\row Ak$ that the condition 
\eqref{e21} holds and also $a^k_{ii}\equiv 0 \!\pmod t$ for all $k,i$.
\item  
 The hereditary algebra $H(\bN)=\End_{R(\bN)}(\rad R(\bN))$ consists of such collections 
$\row Ak$ that 
\emph{
\begin{align*}
  a^k_{ij}\equiv 0 \hskip-1.5ex\pmod t\ & \text{ if } i>j, \text{ except the case when } 
\\	&\	i=j-1 \text{ and } (k,i)\sim(k,i).
\end{align*}	}
\item  
 $M(\bN)$ is a maximal order containing $R(\bN)$ such that $J(\bN)=\rad M(\bN)$ 
is the conductor of $M(\bN)$ both in $R(\bN)$ and in $H(\bN)$, and $J(\bN)\sbe\rad 
R(\bN)$.
\item  
 $R(\bN)/J(\bN)$ is the subring of $M/J(\bN)=\prod_{k=1}^s\Mat(n_k,\aK)$ consisting of 
such collections of matrices $\row As$ that
\begin{align*}
&  a^k_{ij}=0 \,\text{ \em if }\, i>j \,\text{ \em or }\, i=j-1 \,\text{ \em and }\, 
(k,i)\sim(k,i), \\
&  a^k_{ii}=a^l_{jj} \, \text{ \em if }\, (k,i)\sim(l,j). 
\end{align*}
\end{enumerate}
 In particular, $R(\bN)$ is hereditary \iff the relation $\sim$ is empty. 
\emph{(Then we write $R=R(s;\lst ns)$.)}
\end{theorem}

 Actually, to define a ring Morita equivalent to $R(\bN)$, one only has to prescribe 
positive integers $m(k,i)$ for each pair $(k,i)\in\mI$ so that $m(k,i)=m(l,j)$ if 
$(k,i)\sim(l,j)$, and consider $a_{ij}^k$ in the definition of $R(\bN)$ not as 
elements of $\aK$, but as matrices from $\Mat(m(k,i)\xx m(k,j),\aK)$, preserving all 
congruences modulo $t$. We denote such data by $(\bN,\fM)$, where $\fM$ is 
the function $(k,i)\mps m(k,i)$, and the corresponding algebra by $R(\bN,\fM)$. Note that 
different data $\bN$ or $(\bN,\fM)$ can describe isomorphic algebras, even if they do not 
only differ by a permutation of indices $(k,i)$. We extend the relation $\sim$ to an 
equivalence relation $\esim$ setting $(k,i)\esim(l,j)$ \iff $(k,i)=(l,j)$ or 
$(k,i)\sim(l,j)$.

 From the well-known properties of torsion free modules over reduced rings of Krull 
dimension $1$ (see, for instance, \cite{cr}) it follows that, given a torsion free 
coherent sheaf $\kF$ over a \NC\ $(X,\kA)$, a finite set of closed points $\lst xm\in X$ 
and a set of coherent $\kA_{x_i}$-submodules $\kG_i\sb\kF\*_{kO_X}\kK$, there is a 
unique coherent sheaf $\kG\sb\kF\*_{kO_X}\kK$ such that $\kG_{x_i}=\kG_i$ and 
$\kG_y=\kF_y$ if $y\ne x_i$ for all $i$. In particular, since almost all localizations 
$\kA_x$ are maximal, one can construct a normalization $\tA$ of $\kA$ locally, choosing 
arbitrary normalizations $\tA_x$ of $\kA_x$ for $x\in\sng\kA$. Therefore, given a nodal 
\NC\ $(X,\kA)$, we can (and will) suppose that the normalizations of its local components 
are chosen as in Theorem \ref{22}. Thus, if $x\in\sng X$, $y\in\pi^{-1}(x)=\set{\lst 
yr}$, we identify $\tA_y$ with a full matrix ring $\Mat(n_y,\kO_{\tX,y})$ and suppose 
that the ring $\kA_x$ is given by some data $(\bN,\fM)$ as above. In what follows, 
we write $(y_k,i)$ instead of $(k,i)$, so the local embeddings 
$\kA_x\to\tA_x=\prod_{k=1}^r\tA_{y_k}$ for 
$x\in\sng\kA$ are described by the data $\bN(\kA)$ consisting of integers $n_y$ and 
$m(y,i)$ for $y\in\tsg\kA,\ 1\le i\le n_y$, and an equivalence relation $\sim$ on the 
set of pairs $(y,i)$ satisfying the above conditions (N1) and (N2) and such that
\begin{enumerate}
\item[(N3)]\  the sum $\fM_y=\sum_{i=1}^{n_y}m(y,i)$ is the same for all points $y$ 
belonging to the same component of $\tX$.
\end{enumerate}
 The last condition just expresses the fact that the sheaf $\kK(\tA)$ is locally 
constant. One easily sees that $\pi(y)=\pi(y')$ \iff there is at least one relation 
$(y,i)\sim(y',j)$. Moreover, if we suppose that $X$ is connected and $\kA$ is central,
the set $\pi^{-1}(x)$ for each $x\in\sng X$ must be connected as the graph defined by the
symmetric relation $y\sim y'$ which means that there is at least one pair $i,j$ such that 
$(y,i)\sim(y',j)$.

 From now on we fix a connected central \nnc\ $(X,\kA)$ and its normalization 
$\pi:(\tX,\tA)\to(X,\kA)$ chosen as described above. We write $\kO$ instead of 
$\kO_X$ and $\tO$ instead of $\kO_{\tX}$. If $\lst \tX s$ are the irreducible components 
of $\tX$, $X_k=\pi(\tX_k)$, we write $\tO_k=\tO_{\tX_k}$, $\tA_k=\tA|_{\tX_k}$, 
$\kO_k=\kO_{X_k}$ and $\kA_k=\kA|_{X_k}$. Recall that the sheaves of rings $\tO_k$ and 
$\tA_k$ are Morita equivalent. Namely, there is a \VB\ $\kL_k$ over $\tA_k$ such that 
$\END_{\tA_k}\kL_k\simeq\tO_k$, $\END_{\tO_k}\kL_k\simeq\tA_k$, so the functors 
$\Hom_{\tA_k}(\kL_k,\_\,)$ and $\kL_k\*_{\tO_k}\!\!\_$ establish an equivalence between 
$\coh(\tA_k)$ and $\coh(\tO_k)$. We call $\kL_k$ a \emph{basic \VB} over $\tA_k$. (Note 
that it is not uniquely defined.)

 Let $\kJ$ be the conductor of $\tA$ in $\kA$. If $x\in\sng\kA$, then 
$\kJ_x=\bop_{\pi(y)=x}\rad\tA_y$, $\tbS_y=\tA_y/\kJ_y\simeq \Mat(\fM_y,\fK)$ and 
$\kL_y/\kJ_y\kL_y\simeq\fM_yU_y$, where $U_y$ is the unique simple $\tbS_y$-module. For 
any \VB\ $\kG$ over $\tA_i$ we define its \emph{rank}: $\rk\kG=r$ if 
$\kG_y/\kJ_y\kG_y\simeq rU_y$ for some (then for any) $y\in \tX_i$.

 Every pair $(y,i)$, where $\pi(y)=x,\ 1\le i\le n_y$, defines a simple 
$\bS_x$-module $V_{y,i}$, where $\bS_x=\kA_x/\kJ_x$, and $V_{y,i}\simeq V_{y',j}$ \iff 
$(y,i)\esim(y',j)$. Moreover, $U_y\simeq\bop_{i=1}^{n_y}V_{y,i}$ as $\bS_y$-module. 
We denote by $P_{y,i}$ the projective $\bS_y$-module such that $P_{y,i}/\rad 
P_{y,i}\simeq V_{y,i}$. In particular, $P_{y,i}\simeq P_{y',j}$ \iff $(y,i)\esim(y',j)$.

 To describe the category of triples $\kT(\kA)$ it is convenient to introduce new symbols 
$e^y_{ij}$, where $1\le i\le j\le n_y$, and the sets $E^{y,i}_{y',j}$ consisting 
of all $e^z_{i'j'}$ such that one of the following conditions hold:
\begin{itemize}
\item 
  $z=y,\ (y,j')\sim(y',j)$ and either $i=i'$ or $(y,i)\sim(y,i')$;
\item  
 $z=y',\ (y',i')\sim(y,i)$ and either $j=j'$ or $(y',j)\sim(y',j')$. 
\end{itemize}
 We also set $e^y_i=\sum_{(z,j)\esim(y,i)}e^z_{jj}$ and consider the copies $U e^y_{ii}$ 
of the simple modules $U_y$. Then
\begin{align*}
  \tS\*_\bS P_{y,i}&\simeq \bop_{(z,j)\esim(y,i)} U^z_j e^z_{jj},\\
  \End_\bS P_{y,i}&\simeq \begin{cases}
  \aK e^y_i & \text{ if } (y,i)\not\sim(y,j) \text{ for any } j\ne i,\\
  \aK e^y_i\+\aK e^y_{ij} & \text{ if } (y,i)\sim(y,j) \text{ and } i<j,\\
  \aK e^y_i\+\aK e^y_{ji} & \text{ if } (y,i)\sim(y,j) \text{ and } j<i.
\end{cases}\\
\intertext{and, for $(y,i)\not\sim(y',j)$,}
  \hom_\bS(P_{y,i},P_{y',j})&\simeq \bop_{E^{y,i}_{y',j}} \aK e^z_{i'j'}.
\end{align*}
 Under such notations the maps $\tS\*_\bS P_{y,i}\to \tS\*_\bS P_{y',j}$ induced by the 
homomorphisms $P_{y,i}\to P_{y',j}$ as well as the multiplication of homomorphisms are 
given by the ``matrix multiplication'' on the right, i.e. by the rules:
\[
  e^y_{ii'} e^{y'}_{j'j}=\begin{cases}
  0 & \text{if } y\ne y' \text{ or } i'\ne j',\\
  e^y_{ij} &\text{if } y=y' \text{ and } i'=j'.
\end{cases}  
\]

 Let $(\kG,P,\th)$ be a triple from $\kT(\kA)$. Decompose 
$\kG$ and $P$:
\begin{itemize}
\item 
  $\kG=\bop_{k,l} g_{kl}\kG_{kl}$, where $\kG_{kl}$ are nonisomorphic 
indecomposable \VB s over $\tA_k$,
\item  
 $P=\bop_{y,i}p_{y,i} P_{y,i}$.
\end{itemize}
 Set $r_{kl}=\rk\kG_{kl}$. Then the isomorphism $\th:\hpi^*P\to\tio^*\kG$ is given by a 
set $\Th=\setsuch{\Th_y}{y\in\tsg\kA}$ of invertible 
block matrices $\Th_y=(\Th^{y,i}_{kl})$, where $y\in\tsg_k\kA$, the block 
$\Th^{y,i}_{kl}$ has coefficients from $\aK$ and is of size $r_{kl}g_{kl}\xx p_{y,i}$.  
If another triple $(\kG',P',\th')$ is given by the matrices 
$\Th'_y$, a morphism $(\kG,P,\th)\to(\kG',P',\th')$ is given by block matrices 
$\Phi_k=(\Phi^{kl}_{kl'})$ and $\phi_y=(\phi^{y,i}_{y,j})$ such 
that $\Phi_k(y)\Th_y=\Th'_y\phi_y$ for every $y\in\tsg_k\kA$, where the elements of 
$\Phi^{kl}_{kl'}$ are from $\hom_{\tA_k}(\kG_{kl},\kG_{kl'})$, elements of 
$\phi^{y,i}_{y,j}$ are from $\aK$, $\phi^{y,i}_{y,i}=\phi^{y',j}_{y',j}$ if 
$(y,i)\sim(y',j)$ and $\phi^{y,i}_{y,j}=0$ if $i>j$ or $i=j-1,\,(y,i)\sim(y,i)$. This 
morphism is an isomorphism \iff all ``diagonal'' blocks $\Phi^{kl}_{kl}$ and 
$\phi^{y,i}_{y,i}$ are invertible. 

 Let $\kN(\kA)$ be the ideal in $\kT(\kA)$ consisting of all morphisms $(\Phi,\phi)$ such 
that all values $\Phi_k(y)$, where $y\in\tsg_k\kA$, are zero. In the matrix presentation 
it means that $\Phi^{kl}_{kl'}(y)=0$ for all possible triples $(k,l,l')$ and all 
$y\in\tX_k$. Denote $\hT(\kA)=\kT(\kA)/\kN(\kA)$. These categories have the same objects 
and the natural functor $\kT(\kA)\to\hT(\kA)$ is full (not faithful), maps nonisomorphic 
objects to nonisomorphic and indecomposable objects to indecomposable. Therefore, to 
obtain a classification of \VB s, we actually have to study the category $\hT(\kA)$. 
Nevertheless, passing from $\kT$ to $\hT$ we can lose some information. It is 
important, for instance, if we are looking for \emph{stable} \VB s (see, for instance, 
\cite{bod}).

 \section{String case}
\label{s3}

 \begin{defin}\label{31}
  A \nnc\ $(X,\kA)$ is said to be of \emph{string type} if it is rational and every 
set $\tsg_k\kA$ contains at most $2$ points.

 If $(X,\kA)$ is of string type, we identify all components $\tX_k$ with $\mP^1$ and fix 
an affine part $\mA^1\sb\tX_k$ containing $\tsg_k\kA$ .
\end{defin}

 In this case the category of triples $\kT(\kA)$ can be treated as the category of 
representations of a certain \emph{bunch of chains} $\dB(\kA)$ in the sense of 
\cite[Appendix B]{bd}.\!%
\footnote{\,Or a \emph{bundle of semi-chains} in the terms of \cite{bon}.}
 Namely, if $\kL_k$ is a basic vector bundle over $\tA_k$, then every 
indecomposable vector bundle over $\tA_k$ is isomorphic to $\kL_k(d)$ 
for some $d$, which is called the \emph{degree} of $\kL_k(d)$.\!%
\footnote{\,Note that it is not the degree of $\kL_k(d)$ as of $\tO_k$-sheaf; the latter 
equals $dn_k$.}
 Moreover, 
 \[
  \hom_{\tA}(\kL_k(d),\kL_k(d'))\simeq
\begin{cases}
  0 & \text{if } d>d',\\
  \aK[t]_{d'-d} & \text{if } d\le d',
\end{cases}  
\]
 where $\aK[t]_m$ denotes the set of polynomials $f(t)$ such that $\deg f(t)\le m$. 
Therefore, in the decomposition of a \VB\ $\kG_k$ over $\tA_k$ we can suppose that 
$\kG_{kl}=\kL_k(l)$. Then the elements of the matrices $\Phi^{kl}_{kl'}$ can be 
considered as the polynomials of degree $l'-l$ if $l'\ge l$; they are zero if $l'<l$. If 
$y\ne y'$ are two points from $\tsg_k\kA$ and $l'>l$, we can always choose a 
polynomial $f(t)\in\aK[t]_{l'-l}$ such that $f(y)=a,\,f(y')=b$ for any prescribed values 
$a,b\in\aK$. It means that the values of the matrices $\Phi^{kl}_{kl'}$ at the points $y$ 
and $y'$ can be prescribed arbitrary. Therefore, the rule $\Phi_k(y)\Th_y=\Th'_y\phi_y$ 
from the matrix description of morphisms in $\kT(\kA)$ can be rewritten as 
$F(y)\Th_y=\Th'_y\phi_y$, where $F(y)$ is an arbitrary lower block triangular matrix 
$F(y)=(F(y)^{kl}_{kl'})$ ($F(y)^{kl}_{kl'}=0$ if $l<l'$) over the field $\aK$ and the 
only restrictions for these blocks is that $F(y)^{kl}_{kl}=F(y')^{kl}_{kl}$ if $y$ and 
$y'$ are in the same component $\tX_k$.

 Thus we define the bunch of chains $\dB(\kA)$ as follows. We consider $\tsg\kA$ as the 
index set of this bunch and for every $y\in\tsg\kA$ set \label{bunch}
\begin{align*} 
&  \dE_y = \setsuch{(y,i)}{1\le i\le n_y}\setminus\setsuch{(y,i)}{(y,i-1)\sim(y,i-1)},\\
 & \dF_y= \setsuch{(d,y)}{d\in\mZ},\\
&  (y,i)<(y,j)\ \text{ if } i<j, \\
 & (d,y)<(d',y)\  \text{ if } d<d',\\
&  (y,i)\sim(y',j)  \text{ \iff they are so in the nodal data $\bN(\kA)$},\\
&  (d,y)\sim(d',z) \text{ \iff } d=d',\,y\ne z \text{ but $y$ and $z$ belong}\\
& \hskip6.5em \text{to the same component } \tX_k.
 \end{align*}
 Recall \cite{bon,bd} that a representation $M$ of this bunch of chains is given by a set 
of block matrices $M_y=(M_{dy}^{yi})$, where $y\in\tsg\kA,\ 1\le i\le n_y$, 
$M_{dy}^{yi}\in\Mat(m_{dy}\xx n_{yi},\aK)$ for some integers $m_{dy},n_{yi}$ such that 
 $m_{dy}=m_{dy'}$ if $(d,y)\sim(d,y')$ and $n_{yi}=n_{y'j}$ if $(y,i)\sim(y',j)$. Here we 
identify the symbols $(y,i)'$ and $(y,i)''$ from \cite[Definition B.1]{bd}, where 
$(y,i)\sim(y,i)$, with the pairs $(y,i)$ and $(y,i+1)$. A morphism $\al:M\to M'$  
 given by a set of block matrices $\al'_y,\al''_y$, where $y\in\tsg\kA$,  
$\al'_y=(\al^{dy}_{d'y})$, $\al''_y=(\al^{yi}_{yi'})$, such that
\begin{align*}
  \al^{dy}_{d'y}&\in \Mat(m_{d'y}\xx m_{dy},\aK),\\
  \al^{yi}_{yi'}&\in \Mat(m_{yi'}\xx m_{yi},\aK),\\
  \al^{dy}_{d'y}&=0\ \text{ if } d>d',\\
  \al^{yi}_{yi'}&=0\ \text{ if } i>i' \text{ or } i'=i+1 \text{ and } (y,i)\sim(y,i),\\
  \al^{dy}_{dy}&=\al^{dy'}_{dy'}\ \text{ if } (d,y)\sim(d,y'),\\
  \al^{yi}_{yi}&=\al^{y'j}_{y'j}\ \text{ if } (y,i)\sim(y',j),\\
\intertext{ and }
 \al'_yM_y&=M'_y\al''_y \ \text{ for all } y\in\tsg\kA.
\end{align*}

 The matrix presentations described above imply the following fact.

\begin{prop}\label{32}
 Let the \nnc\ $(X,\kA)$ is of string type, $\dB=\dB(\kA)$. Then the category 
$\hT(\kA)$ is equivalent to the full subcategory $\rep_0(\dB)$ of the category of 
representations of the bunch of chains $\dB$ consisting of such representations $M$ that 
all matrices $M_y$ are invertible.
\end{prop}

 In particular, the category $\hT(\kA)$ and hence the category $\vb(\kA)$ are \emph{tame} 
in the sense that they have at most $1$-parameter families of indecomposable objects. 
Moreover, from the description of representations of a bunch of chains given in 
\cite{bon} one can deduce a description of \VB s over a \nnc\ of string type. For the 
corresponding combinatorics we use the terminology from \cite{bd} adopted to our 
situation.

 \begin{defin}\label{33}
\begin{enumerate}
\item 
 Let $\dE=\bigcup_y\dE_y,\,\dF=\bigcup_y\dF_y,\,\dX=\dE\cup\dF$. We define the symmetric 
relation $-$ on $\dX$ setting $(d,y)-(y,i)$ for all possible $d,i,y$. We also write 
$\xi\parallel \xi'$ if either both $\xi$ and $\xi'$ belong to $\dE$ or both of them 
belong to $\dF$, and $\xi\perp \xi'$ if one of them belongs to $\dE$ while the other 
belongs to $\dF$.
\item  
 We define a \emph{word} (more precisely, an \emph{$\dX$-word}) as a sequence 
$\xi_1r_1\xi_2r_2\dots \xi_{l-1}r_{l-1}\xi_l$ such that	
	\begin{enumerate}
	\item   $\xi_i\in\dX$, $r_i\in\set{\sim,-}$;
	\item  
 	$\xi_ir_i\xi_{i+1}$ for each $1\le i<l$ accordingly to the definition of the relations 
	$\sim$ and $-$;	
	\item  
 	$r_i\ne r_{i+1}$ for all $1\le i<l-1$.
	\end{enumerate}
 We call $l=l(w)$ the \emph{length} of the word $w$ and $\xi_1,\xi_l$ the \emph{ends} of 
this word.
\item  
 We call the word $w$ \emph{full} if the following conditions hold:
	\begin{enumerate} 
	\item 
	either $r_1=\sim$ or $\xi_1\not\sim \xi'$ for any $\xi'\ne \xi_1$;
	\item 
	either $r_{l-1}=\sim$ or $\xi_l\not\sim \xi'$ for any $\xi'\ne \xi_l$.
	\end{enumerate}
\item  
 We call the word $w$ \emph{terminating} if it is full and $r_1=r_{l-1}=-$.
\item  
 The end $\xi_1$ ($\xi_l$) is said to be \emph{special} if $r_1=-$ and $\xi_1\sim\xi_1$ 
(respectively, $\xi_l\sim\xi_l$ and $r_{l-1}=-$). Otherwise it is said to be \emph{usual}.
\item  
 The terminating word $w$ is said to be
	\begin{itemize}
	\item  
 	\emph{usual} if both its ends are usual;
	\item 
	 \emph{special} if one of its ends, but not both, is special; 
	\item  
 	\emph{bispecial} if both its ends are special.
	\end{itemize}
\item  
 The word $w^*=\xi_lr_{l-1}\dots \xi_2r_1\xi_1$ is called \emph{inverse} to the word $w$. 
\item  
 We call $w$ \emph{symmetric} if $w=w^*$ and \emph{quasisymmetric} if it can be presented 
as $v\sim v^*\sim v\sim\dots\sim v^*\sim v$ for a shorter word $v$. Note that a 
quasisymmetric word is always bispecial.
\item  
 The word $w$ is said to be \emph{cyclic} if $r_1=r_{l-1}=\sim$ and $\xi_l-\xi_1$ in 
$\dB$. Then we set $r_0=-$ and $\xi_{i+kl}=\xi_i,\,r_{i+kl}=r_i$ for any $k\in\mZ$.
\item  
 A \emph{shift} of the cyclic word $w$ is the cyclic word 
$$w^{[k]}=\xi_{k+1}r_{k+1}\xi_{k+2}\dots r_0\xi_1r_1\dots \xi_k,$$
 where $k$ is even. In this case we set $\eps(w,k)=(-1)^{k/2}$.
\item  
 The cyclic word $w$ is said to be \emph{aperiodic} if $w^{[k]}\ne w$ for $0<k<l$. It is 
said to be \emph{cyclic-symmetric} if $w^*=w^{[k]}$ for some $k$.
 \end{enumerate}
 Note that the length of a terminating or cyclic word is always divisible by $4$.
\end{defin}

 \begin{defin}\label{34}
  \begin{enumerate}
 \item 
 A \emph{usual string} is a usual nonsymmetric terminating word.
\item  
 A \emph{special string} is a pair $(w,\de)$, where $w$ is a special terminating 
word and $\de\in\set{0,1}$.
\item  
 A \emph{bispecial string} is a quadruple $(w,m,\de_0,\de_1)$, where $w$ is a 
bispecial terminating word that is neither symmetric nor quasisymetric, $m\in\mN$ and 
$\de_i\in\set{0,1}\ (i=0,1)$.
\item  
 A \emph{band} is a triple $(w,m,\la)$, where $w$ is a cyclic word, $m\in\mN$, 
$\la\in\aK^\xx$ and, if $w$ is cyclic-symmetric, also $\la\ne1$.
 \item  
 The following strings are said to be \emph{equivalent}:
	\begin{itemize}
	\item 
	$w$ and $w^*$;
	\item  
 	$(w,\de)$ and $(w^*,\de)$;
	\item  
 	$(w,m,\de_0,\de_1)$ and $(w^*,m,\de_1,\de_0)$.
	\end{itemize}
\item  
 Two bands are said to be \emph{equivalent} if they can be obtained from one another 
by a sequence of the following transformations:
	\begin{itemize}
	\item 
	replacing $(w,m,\la)$ by $(w^{[k]},m,\la^{\eps(w,k)})$;
	\item  
 	replacing $(w,m,\la)$ by $(w^*,m,\la^{-1})$.
	\end{itemize}
 Note that if $w^*=w^{[k]}$, then $k\equiv2\!\pmod4$, so $\eps(w,k)=-1$.
\end{enumerate}
\end{defin}

 Now the results of \cite{bon} imply the following theorem.

\begin{theorem} \label{35}
  The isomorphism classes of indecomposable \VB s over a \nnc\ curve of string type 
$(X,\kA)$ are in \oc\ with the equivalence classes of strings and bands for the bunch of 
chains $\dB(\kA)$. The rank of the \VB\ corresponding to a string or a band equals $l/4$, 
where $l$ is the length of the word $w$ entering into this string or band.
\end{theorem}

 We refer to \cite{bon} for an explicit construction of representations corresponding 
to strings and bands, hence of \VB s over \nnc s of string type.

 Note that it can so happen that there are no strings or no bands. For instance, 
if all localizations $\kA_x$ are hereditary, there are no bands as well as no special and 
bispecial strings. Then there are only finitely many isomorphism classes of 
indecomposable \VB s \emph{up to twist}, i.e. up to change of degrees $d$ in the pairs 
$(d,y)$ occurring in a word. On the other hand, if each $\tsg_k\kA$ consists of $2$ points 
and for every pair $(y,i)$ there is another pair $(z,j)\ne(y,i)$ such that 
$(z,j)\sim(y,i)$, then there are no terminating strings.

Actually, one can easily deduce the following criterion of finiteness.

\begin{corol}\label{36}
  The following conditions for a \nnc\ of string type $(X,\kA)$ are equivalent:
\begin{enumerate}
\item 
 There  are only finitely many isomorphism classes of indecomposable \VB s over $\kA$ up 
to twist.
\item  
 There are no cycles for the bunch of chains $\dB(\kA)$.
\item  
 There are no sequences of points $\lst yn,y_{n+1}=y_1$ from $\tsg\kA$ such that, for 
$1\le k\le n$, 
	\begin{itemize}
	\item 
  	if $k$ is odd, then the points $y_k$ and $y_{k+1}$ are different and belong to the 
	same component of $\tX$;	
	\item  
 	if $k$ is even, there are indices $i,j$ such that $(y_k,i)\sim(y_{k+1},j)$ \lb possibly 
	$y_k=y_{k+1}$\rb.
	\end{itemize}
\end{enumerate}
\end{corol}

\section{Almost string case}
\label{s4}

We consider one more case when there is a good description of \VB s.

\begin{defin}
  A \nnc\ $(X,\kA)$ is said to be of \emph{almost string type} if every set $\tsg_k\kA$ 
contains at most $3$ point, and if it contains three points then for $2$ of them the 
algebra $\kA_{\pi(y)}$ is hereditary and Morita equivalent to the algebra $R(1;2)$ from 
Theorem \ref{22} (with the empty relation $\sim$).
\end{defin} 

 Note that if $\kA_{\pi(y)}$ is hereditary, $y$ is the unique point of $\tsg\kA$ with the 
image $\pi(y)$. Hence, if $X$ is connected, either $\tX$ consists of a unique component 
or there must be another point $z$ on the same component of $\tX$ such that 
$\kA_{\pi(z)}$ is not hereditary.

 Let $\tsg_k\kA=\set{y_0,y_1,y_2}$ so that $\kA_{\pi(y_1)}$ and $\kA_{\pi(y_2)}$ are 
Morita equivalent to $R(1;2)$. In this case we call $y_1,y_2$ \emph{extra points} and 
$y_0$ a \emph{marked point}. Then the horizontal stripes of the matrices 
$\Th_{y_1},\Th_{y_2}$ corresponding to the \VB\ $\kL_k(d)$ can be reduced to the form
\begin{equation}\label{pre}\scriptstyle
  {\textstyle\Th_{kd}^{y_1,1}}=\mtr{0&0\\I&0\\0&0\\0&I}\!,\ 
  {\textstyle\ \Th_{kd}^{y_1,2}}=\mtr{I&0\\0&0\\0&I\\0&0}\!,\ 
  {\textstyle\ \Th_{kd}^{y_2,1}}=\mtr{0&0\\0&0\\I&0\\0&I}\!,\ 
  {\textstyle\ \Th_{kd}^{y_2,2}}=\mtr{I&0\\0&I\\0&0\\0&0}\!,
\end{equation}	
 where $I$ denote identity matrices of some sizes (equal if they are in the same row).
 From now on we only consider the objects from $\kT(\kA)$ such that these matrices have 
the form \eqref{pre}, calling them \emph{precanonical}. If $(\Phi,\phi)$ is a morphism 
between precanonical objects, then the matrix $\Phi^{kd}_{kd}$ must be of  the $4\xx4$ 
block form
\[
  \Phi^{kd}_{kd}=\mtr{*&0&0&0\\ *&*&0&0\\ *&0&*&0\\ *&*&*&*}\!,  
\]
  where stars denote arbitrary matrices of appropriate sizes. Moreover, if we consider 
$\Phi^{k,d-1}_{kd}$ also as a $4\xx4$ block matrix $(f_{ab})\ (a,b\in\set{1,2,3,4})$, 
where the blocks $f_{ab}$ consist of linear polynomials, then 
$f_{14}(y_1)=f_{14}(y_2)=0$, so $f_{14}=0$. Note that the values $f_{ab}(y_0)$ can be 
chosen arbitrary for $(ab)\ne(14)$, as well as the values of $\Phi^{kd'}_{kd}$ for 
$d'<d-1$. Therefore, the full subcategory of $\hT(\kA)$ consisting of precanonical 
objects can again be treated as the category of representations of a bunch of chains 
$\dB'=\dB'(\kA)$. Namely, let $\et\kA$ be the set of all extra points. The index set for 
the bunch $\dB'$ is $\tsg\kA\setminus\et\kA$. If a point $y$ is not marked, the sets 
$\dE_y$ and $\dF_y$ are defined just as in Section \ref{s3} (page \pageref{bunch}). If 
$y$ is marked, the set $\dF_y$ is also defined as in Section \ref{s3}, but the set 
$\dE_y$ consists of the triples $(d,y,\al)$, where $\al\in\set{0,1}$, such that 
\begin{itemize}
\item 
  $(d',y,\al')<(d,y,\al)$ \iff either $d'<d$ or $d'=d$ and $\al'<\al$;
\item  
$(d,y,\al)\sim(d,y,\al)$ for all $d,\al$. 
\end{itemize}
 Actually the element $(d,y,0)$ represents in this bunch the first horizontal row of 
the stripe $(d,y)$ and the fourth horizontal row of the stripe $(d-1,y)$ in the 
precanonical form \eqref{pre}, while the element $(d,y,1)$ represents the second and 
the third horizontal rows of the stripe $(d,y)$.

 The preceding observations imply 
\begin{theorem}\label{41}
 Let $(X,\kA)$ be a \nnc\ of almost string type. The category $\hT(\kA)$ is equivalent to 
the full subcategory of the category of representations of the bunch of chains 
$\dB'(\kA)$ consisting of such representations $M$ that all matrices $M_y$ are invertible.
\end{theorem}

 Just as in Section \ref{s3}, these representations (hence, \VB s over $\kA$) correspond 
to terminating strings and bands. In particular, the category of \VB s over a \nnc\ of 
almost string type is also tame.

\begin{corol}\label{42}
    The following conditions for a \nnc\ of almost string type $(X,\kA)$ are equivalent:
\begin{enumerate}
\item 
 There  are only finitely many isomorphism classes of indecomposable \VB s over $\kA$ up 
to twist.
\item  
 There are no cycles for the bunch of chains $\dB'(\kA)$.
\item  
 There are no sequences of points $\lst yn,y_{n+1}=y_1$ from $\tsg\kA\setminus\et\kA$ 
such that, for $1\le k\le n$, 
	\begin{itemize}
	\item 
  	if $k$ is odd, then either the points $y_k$ and $y_{k+1}$ are different and belong to 
	the same component of $\tX$ or $y_k=y_{k+1}$ is a marked point;	
	\item  
 	if $k$ is even, there are indices $i,j$ such that $(y_k,i)\sim(y_{k+1},j)$ \lb possibly 
	$y_k=y_{k+1}$\rb.
	\end{itemize}
\end{enumerate}
\end{corol}

 \section{Wild cases}
\label{s5}

 If a \NC\ $(X,\kA)$ is rational and connected and all localizations $\kA_x$ are 
hereditary, then $X\simeq\mP^1$ and the category $\coh(\kA)$ is equivalent to the 
category of coherent sheaves over a \emph{weighted projective line} $C(\fP,\bla)$ in the 
sense of \cite{gl}. Here $\bla=\set{\lst\la s}=\sng\kA$ and $\fP=\row ps$ 
are the integers such that $\kA_{\la_k}$ is Morita equivalent to the hereditary algebra 
$R(1;p_k)$. Then it is known that $\vb(\kA)$ is of finite type \iff $\sum_{k=1}^s1/p_k>1$ 
and is tame if $\sum_{k=1}^s1/p_k=1$. If $\sum_{k=1}^s1/p_k<1$, it is \emph{wild}. It 
means that the classification of \VB\ over such \NC\ contains the classification of 
representations of every finitely generated $\aK$-algebra (see \cite{dg} for formal 
definitions). Note also that if $(X,\kA)$ is normal, then, just as $X$ itself, it is of 
finite type if $X\simeq\mP^1$, tame if $X$ is an elliptic curve and wild otherwise 
\cite{dg}. So the next theorem completes the answer to the question about the 
representation type of $\vb(\kA)$.

\begin{theorem}\label{51}
  In the following cases the category $\vb(\kA)$ is wild:
\begin{enumerate}
\item 
  $(X,\kA)$ is neither rational nor normal.
\item  
  At least one of the localizations $\kA_x$ is not nodal.
\item  
  $(X,\kA)$ is nodal, at least one of the localizations $\kA_x$ is not hereditary and 
$(X,\kA)$ is neither of string nor of almost string type.
\end{enumerate}
\end{theorem}
\begin{proof}
  The cases (1) and (2) are considered quite analogously to the commutative case 
\cite[Proposition 2.5]{dg}, so we omit their proofs. The proof of (3) we shall give in 
two cases:
\begin{itemize}
\item[(3a)]
 $X=\mP^1$, $\sng\kA=\set{x,x_2,x_3}$, $\kA_{x_k}$ is Morita equivalent to $R(1;k)$ 
for $k=2,3$, while, $\kA_x$ is Morita equivalent to $R(1;2;\sim)$, where either 
$(1,1)\sim(1,2)$ or $(1,1)\sim(1,1)$. 
\item[(3b)]  
 $X=X_1\cup X_2$ so that $X_1\simeq X_2\simeq\mP^1$, $X_1\cap X_2=\{x\}$ and this 
intersection is transversal (i.e. $\kO_x$ is nodal), there are two more singular points 
$x_2,x_3\in X_1$ and $\kA_{x_k}$ is Morita equivalent to $R(1;k)$ for $k=2,3$, while 
$\kA_x$ is Morita equivalent to ${R(2;1,1;\sim)}$, where $(1,1)\sim(2,1)$.
\end{itemize}  
 All other cases easily reduce to these ones.

 In both cases $\pi^{-1}(x_k)=\{y_k\}$ for $k=2,3$ and the $d$-th 
horizontal stripe of the matrices $\Th_{y_k}$ can be reduced to the form:
\[
 \Th_{2d}=\left(\begin{array}{cc|cc|cc}
  0&0&0&0&I&0\\ 0&0&I&0&0&0\\ I&0&0&0&0&0\\
  0&0&0&0&0&I\\ 0&0&0&I&0&0\\ 0&I&0&0&0&0
\end{array}\right)\!,\ 
  \Th_{3d}=\left(\begin{array}{ccc|ccc}
  0&0&0&I&0&0\\ 0&0&0&0&I&0\\ 0&0&0&0&0&I\\
  I&0&0&0&0&0\\ 0&I&0&0&0&0\\ 0&0&I&0&0&0\\
\end{array}\right)\!,
\]
 where the vertical lines divide these matrices into the stripes corresponding to the 
projective modules $P_{ki}$. In the case (3a) we only consider such triples that 
the 1st, 5th and 6th horizontal rows of these matrices are empty. Then the matrix 
$\Th_y$, where $y\in\tsg_1\kA$ and $\pi(y)=x$, is divided into $3$ horizontal stripes and 
if $(\Phi,\phi)$ is a morphism of such representations, then 
\[
  \Phi^{1d}_{1d} =\mtr{*&0&0\\ *&*&0\\ 0&0&*}.
\]
 The classification of such triples can be considered as a bimodule problem (see 
\cite{ds,dg} for definitions and details) so that the corresponding Tits form is either 
\begin{align*}
 Q_1&= 2t_1^2+z_1^2+z_2^2+z_1z_2+z_3^2-2t_1(z_1+z_2+z_3)\\
\intertext{or}
 Q_2&= t_1^2+t_2^2+z_1^2+z_2^2+z_1z_2+z_3^2-(t_1+t_2)(z_1+z_2+z_3),
\end{align*} 
 where $t_i$ are the sizes of vertical stripes and $z_i$ are the sizes of horizontal 
stripes (if $(1,1)\sim(1,2)$, then $t_1=t_2$). Since $Q_1(2,1,1,1)=Q_2(2,2,1,1,1)=-1$, 
this bimodule is wild, hence so is the category $\vb(\kA)$. Note that we need to check 
that $t_1+t_2=z_1+z_2+z_3$, since the matrix $\Th_y$ must be invertible.

 In the case (3b) we only omit the 1st and the 6th row of the matrices $\Th_{y_k}$. Then 
the matrix $\Phi_{1d}^{1d}$ will be of the form
\[
  \Phi^{1d}_{1d}=\mtr{*&0&0&0\\ *&*&0&0\\ 0&0&*&0\\ *&0&*&*}
\]
We have one more matrix $\Th_z$, where $z\in\tsg_2\kA$ and $\pi(z)=x$. 
We consider the triples such that $\kG|_{Y_2}=\bop_{d=1}^8r_d\kG_{2d}$. The matrix 
$\Th_z$ reduces to the form
\[
  \Th_z=\left(\begin{array}{c|c|c|c|c|c|c|c}
  I&0&0&0&0&0&0&0\\ 0&I&0&0&0&0&0&0\\ 0&0&I&0&0&0&0&0\\ 0&0&0&I&0&0&0&0\\
   0&0&0&0&I&0&0&0\\ 0&0&0&0&0&I&0&0\\ 0&0&0&0&0&0&I&0\\ 0&0&0&0&0&0&0&I
\end{array}\right)
\]
 Then the matrix $\phi^{y,1}_{y,1}=\phi^{z,1}_{z,1}$ from a morphism $(\Phi,\phi)$ of 
such triples must be triangular and we obtain a matrix problem with the Tits form
\[
 Q= t_1^2+t_2^2+t_3^2+t_4^2+t_1t_2+t_1t_4+t_3t_4+\sum_{i\le j}r_ir_j- \sum_{i,j}t_ir_j.  
\]
 Now $Q(1,3,3,1,1,1,1,1,1,1,1,1)=-1$, so we again obtain a wild problem.
\end{proof}

\section{Example}

 We consider a simple but typical example. Let $(X,\kA)$ be defined as follows.
\begin{itemize}
\item 
 $X=X_1\cup X_2$ , where $X_1\simeq X_2\simeq \mP^1$, $X_1\cap X_2=\{x\}$ and the 
intersection is transversal;
\item  
 $\sng\kA=\set{x,x_1,x_2}$, where $x_1\in X_1,\,x_2\in X_2$;
\item  
 $\kK(\kA)=\Mat(2,\kK_1)\xx\Mat(2,\kK_2)$;
\item  
 The singular localizations are:
\begin{align*}
  \kA_x&=R(2;2,2;\sim), \text{ where } (1,1)\sim(2,1);\\
  \kA_{x_1}&=R(1;2;\sim), \text{ where } (1,1)\sim(1,1);\\
  \kA_{x_2}&=R(1;2;\sim), \text{ where } (1,1)\sim(1,2).
\end{align*}
\end{itemize}
 Then 
\begin{itemize}
\item 
 $\tX=\tX_1\cup\tX_2$, where $X_1\simeq X_2\simeq \mP^1$, $X_1\cap X_2=\emptyset$;
\item  
 $\tsg\kA=\set{y_1,y_2,y_3,y_4}$, where $y_1,y_3\in Y_1,\,y_2,y_4\in Y_2$, \\ 
$\pi(y_3)=\pi(y_4)=x,\,\pi(y_1)=x_1,\,\pi(y_2)=x_2$.
\end{itemize}
 Therefore the corresponding bunch of chains is
\begin{align*}
  \dE_1&=\setsuch{(d1)}{d\in\mZ},\ \dF_1=\set{(1,1)},\\
  \dE_2&=\setsuch{(d2)}{d\in\mZ},\ \dF_2=\set{(2,1)<(2,2)},\\
  \dE_3&=\setsuch{(d3)}{d\in\mZ},\ \dF_3=\set{(3,1)<(3,2)},\\
  \dE_4&=\setsuch{(d4)}{d\in\mZ},\ \dF_4=\set{(4,1)<(4,2)},\\
  (1,1)&\sim(1,1),\ (2,2)\sim(2,1),\ (3,1)\sim(4,1),\ (d1)\sim(d3),\ (d2)\sim(d4).
\end{align*}
 (We write $(dk)$ and $(k,i)$ instead of $(d,y_k)$ and $(y_k,i)$.) 
 We fix a basic vector bundle $\kL_k$ over $\tA_k$ ($k=1,2$). Then 
$\kL_1(d)/\kJ\kL_1(d)$ has a $\aK$-basis $e^1_i(d),e^3_j(d),\ (1\le i,j\le2)$ and 
$\kL_2(d)/\kJ\kL_2(d)$ has a $\aK$-basis $e^2_i(d),e^4_j(d),\ (1\le i,j\le 2)$, the 
upper index showing the point $y_k$ where the corresponding element is supported.

 An example of a usual string is given by the word
\[
  (4,2)-(d_14)\sim(d_12)-(2,2)\sim(2,1)-(d_22)\sim(d_24)-(4,2)  
\]
 with $d_1\ne d_2$ in order that the word be not symmetric.
 The corresponding \VB\ $\kF$ is the $\kA$-submodule in $\kG=\kL_2(d_1)\+\kL_2(d_2)$ such 
that $\kF_x=\kG_x$ for $x\notin\sng\kA$, $\kF_{x_2}$ is generated by the 
preimages of $e^2_2(d_1)$ and $e^2_1(d_2)$, and $\kF_{x}$ is generated by the preimages 
of $e^4_2(d_1)$ and $e^4_2(d_2)$. Since $\supp\kG=X_2$, $\kF_{x_1}=0$.

 An example of a special string is $(w,1)$, where.  
\[
  w=(1,1)-(d1)\sim(d3)-(3,2).
\]
 Here $\kG=\kL_1(d)$, $\kF_{x_1}$ is generated by the preimage of $e^1_2$ and $\kF_x$ is 
generated by the preimage of $e^3_2$.

 An example of a bispecial string is $(w,m,1,0)$, where
\begin{align*}
  w&=(1,1)-(d_11)\sim(d_13)-(3,1)\sim(4,1)-(d_24)\sim(d_22)-(2,1)\sim\\
 &\sim(2,2)- (d_32)\sim(d_34)-(4,1)\sim(3,1)-(d_43)\sim(d_41)-(1,1).
\end{align*}
 The degrees $d_i$ can be arbitrary with the only restriction that $d_2\ne d_3$ or 
$d_1\ne d_4$. 
\begin{itemize}
\item 
 $\kG=m(\kL_1(d_1)\+\kL_2(d_2)\+\kL_2(d_3)\+\kL_1(d_4))$;
\item  
 $\kF_{x}$ is generated by the preimages of the columns of the matrices 
$I_me^3_1(d_1),\,I_me^3_1(d_4),\,I_me^4_1(d_2)$ and $I_me^4_1(d_3)$, where $I_m$ denotes 
the identity $m\xx m$ matrix;
\item  
 $\kF_{x_2}$ is generated by the preimages of the columns of the matrices 
$I_me^2_1(d_2)$ and $I_me^2_2(d_3)$;
\item  
 $\kF_{x_1}$ is generated by the preimages of the columns of the matrices 
\[
\hskip2em   \mtr{I_q\\0}\!e^1_2(d_1), \  \mtr{0\\I_{m-q}}\!e^1_1(d_1),\ 
  \mtr{I_q\\A_q}\!e^1_1(d_4)\ \text{ and } \mtr{B_q\\I_{m-q}}\!e^1_2(d_4),
\]
 where $q=[(m+1)/2]$ and
\begin{itemize}
\item   
  if $m=2q$, then $A_q=I_q$, $B_q=J_q(0)$, the Jordan $q\xx q$ matrix with eigenvalue $0$;
\item  
 if $m=2q-1$, then $A_q$ is of size $(q-1)\xx q$ and $B_q$ is of size $q\xx(q-1)$, namely,
\[
\hskip2em   A_q=\mtr{1&0&0&\dots &0&0\\ 0&1&0&\dots &0&0\\
	\hdotsfor6 \\ 0&0&0&\dots&1&0}\!,\quad
  B_q=\mtr{0&0&\dots&0\\ 1&0&\dots &0\\ 0&1&\dots& 0\\
	\hdotsfor4 \\ 0&0&\dots&1}\!.
\]
\end{itemize}
\end{itemize}

 Finally, an example of a band is $(w,m,\la)$, where 
\begin{align*}
  w&=(2,2)\sim(2,1)-(d_12)\sim(d_14)-(4,1)\sim(3,1)-(d_23)\sim(d_21)- \\  
  &-(1,1)\sim(1,1)-(d_31)\sim(d_33)-(3,1)\sim(4,1)-(d_44)\sim(d_42).
\end{align*}
 We suppose that $d_3<d_2$ or $d_3=d_2,\,d_4\le d_1$. Then
\begin{itemize}
\item 
  $\kG=m(\kL_1(d_1)\+\kL_2(d_2)\+\kL_2(d_3)\+\kL_1(d_4))$;
\item  
 $\kF_{x_1}$ is generated by the preimages of the columns of the matrices
\[
    \mtr{I_me^1_1(d_2)\\I_me^1_1(d_3)} \text{ and } \mtr{0\\I_me^1_2(d_3)}\!;
\]
\item  
 $\kF_x$ is generated by the preimages of the columns of the matrices $I_me^4_1(d_1),\, 
I_me^3_1(d_2),\, I_me^3_1(d_3)$ and $I_me^4_1(d_4)$;
\item  
 $\kF_{x_2}$ is generated by the preimages of the columns of the matrices $I_me^2_1(d_1)$ 
and $J_m(\la)e^2_2(d_4)$ (the Jordan $m\xx m$ matrix with eigenvalue $\la$).
\end{itemize}
 If $d_2<d_3$ or $d_2=d_3,\,d_1<d_4$, one has to permute $d_2$ and $d_3$ in the 
generators of $\kF_{x_1}$, also permuting the rows.


\begin{thebibliography}{99}
\bibitem{at}
M. Atiyah. Vector bundles over an elliptic curve. Proc. London Math. Soc. 7
(1957), 414--452.

\bibitem{bi}
G. Birkhoff. A theorem on matrices of analytic functions. Math. Ann.
74 (1913), 122--133.

\bibitem{bod}
L. Bodnarchuk, Y. Drozd.
Stable vector bundles over cuspidal cubics. Central Eurpean J.Math. 1 (2003), 650--660.

\bibitem{bon}
{\selectlanguage{ukrainian}
 В. М. Бондаренко. Представления связок полуцепных множеств и их приложения.
 Алгебра и анализ 3, № 5 (1991), 38--61. } (English translation:
\mbox{V. M. Bondarenko}. Representations of bundles of semi-chains and their applications.
 St. Petersburg Math. J. 3 (1992), 973--996.)

\bibitem{bd}
 I. Burban and Y. Drozd. 
 Derived categories of nodal algebras. J. Algebra 272 (2004) 46--94.

\bibitem{bd2}
 I. Burban and Y. Drozd. 
Coherent sheaves on rational curves with simple double points and transversal 
intersections. Duke Math. J. 121 (2004) 189--229.

\bibitem{cr}
 C.W. Curtis and I. Reiner. Methods of Representation Theory with Applications to Finite 
Groups and Orders. Vol. I. John Wiley \& Sons, 1981.

\bibitem{ds}
 Y. Drozd.
Reduction algorithm and representations of boxes and algebras. Comtes Rendue Math. Acad. 
Sci. Canada, 23 (2001), 97--125.

\bibitem{dg}
 Y. Drozd and G.-M. Greuel. 
Tame and wild projective curves and classification of vector bundles. J. Algebra, 246 
(2001) 1--54. 

\bibitem{gl}
 W. Geigle and H. Lenzing.
 A class of weighted projective curves arising in representation theory of finite 
dimensional algebras.
 \emph{Singularities, Representations and Vecor Bundles.}
 Lecture Notes in Math. 1273 (1987), 265--297.

\bibitem{gr}
 A. Grothendieck. Sur la classification des fibr\'es holomorphes sur la 
sph\`ere de Riemann. Amer. J. Math. 79 (1956), 121--138.

\bibitem{mil}
 J. Milnor. Introduction to Algebraic $K$-theory. Princeton Univ. Press, 1971.
{\selectlanguage{ukrainian}
 (Русский перевод: Дж. Милнор. Введение в алгебраическую $K$-теорию.
 Москва: Мир, 1974.)}

\bibitem{ser}
J.-P. Serre. Cohomologie Galoisienne. Springer--Verlag, 1964.

\bibitem{vol}
{\selectlanguage{ukrainian}
 Д. Є. Волошин. Будова нодальних алгебр.
 Укр. матем. ж.  63, № 7 (2011), 880-888.
}

\end{thebibliography}
\end{document}